\documentclass{birkjour}

\usepackage[utf8]{inputenc} 

\usepackage[T1]{fontenc}

\usepackage{microtype} 
 
\usepackage{enumerate}
 
\usepackage{array}

\theoremstyle{plain}
\newtheorem{Theorem}{Theorem}[section]
\newtheorem{Lemma}[Theorem]{Lemma}
\newtheorem{Corollary}[Theorem]{Corollary}
\newtheorem{Proposition}[Theorem]{Proposition}

\theoremstyle{definition}
\newtheorem{Definition}[Theorem]{Definition}

\theoremstyle{remark}

\newtheorem{Remark}[Theorem]{Remark}

\newcommand{\field}[1]{\mathbb{#1}}     

\newcommand{\Fq}{\field{F}_{\!q}}       
\newcommand{\Pj}{\field{P}}             
\newcommand{\R}{\field{R}}              
\newcommand{\Z}{\field{Z}}              

\newcommand{\fdg}{\;|\;}                

\DeclareMathOperator{\spec}{spec}

\begin{document}

\title[Another proof of Grothendieck's theorem]{Another proof of Grothendieck's theorem on the splitting of vector bundles on the projective line}

\author{Claudia Schoemann}
\address{%
University of Goettingen\\
Mathematical Institute\\
Bunsenstr. 3-5\\
D 37073 Goettingen}
\email{claudia.schoemann@mathematik.uni-goettingen.de}

\author{Stefan Wiedmann}
\address{%
University of Goettingen\\
Mathematical Institute\\
Bunsenstr. 3-5\\
D 37073 Goettingen}
\email{stefan.wiedmann@mathematik.uni-goettingen.de}

\thanks{We thank U.~Stuhler for discussions and historical remarks}



\begin{abstract}
This note contains another proof of Grothendieck`s theorem on the splitting of vector bundles on the projective line over a field $k$. Actually the proof is formulated entirely in the classical terms of a lattice $\Lambda \cong k[T]^d$, discretely embedded into the vector space $V \cong K_\infty^d$, where $K_\infty \cong k((1/T))$ is the completion of the field of rational functions $k(T)$ at the place $\infty$ with the usual valuation.
\end{abstract}

\maketitle

\section{Historical Remarks}
In a recent paper \cite{Gekeler}
mainly devoted to various questions connected with the theory of Drinfeld modules, Ernst-Ulrich Gekeler building on some results of Yuichiro Taguchi \cite{Taguchi}
showed the following theorem concerning the reduction theory of lattices over the polynomial ring $A=\Fq[T]$ over the finite field $\Fq$, which we are going to describe below.

Denoting $K_\infty$ the completion of the rational function field $K=\Fq(T)$ with respect to the valuation $|\;\;|_\infty$ at the place $\infty$ and $C_\infty$ a completion of an algebraic closure of $K_\infty$ with respect to the unique extension of the valuation $|\;\;|_\infty$, one has the following result: (Prop.~3.1 of loc.~cit.) Each $A$-lattice $\Lambda \subseteq C_\infty$ of rank $r   > 0$ has a successive minimum basis (SMB loc.~cit.) $\omega_1, \ldots, \omega_r \in \Lambda$ such that additionally 
$\{\omega_1, \ldots, \omega_r\}$ is an $A$-basis of the lattice $\Lambda$.

This is in rather striking contrast to the situation of classical reduction theory for $\Z$-lattices $\Lambda$ in a Euclidian vector space $V = \R^r$, where at least from $r \geq 4$ on (see \cite{Martinet}, p.~51, for a discussion) an SMB (as a basis of the vector space $V$) usually will not be a basis of the integer lattice under consideration.

The reason for this different situation eventually is (a slightly extended version of) Grothendieck's splitting theorem for vector bundles over the projective line $\field{P}^1$. On the other hand, the corresponding statement on the reduction of lattices for function fields of higher genus $g \geq 1$ would be also completely wrong.

Actually below, making use of Gekeler's proof in \cite{Gekeler}, we give yet another proof of an extended version of Grothendieck's splitting theorem. We have freed the context from the field $C_\infty$ and are in fact working with the polynomial ring $A = k[T]$ for an arbitrary field $k$ and with arbitrary $A$-lattices in a finite dimensional vector space $V$ over $K_\infty$, the completion of $K = k(T)$ with respect to the valuation at $\infty$. In particular, different from \cite{Gekeler}, we will not need the local compactness of $K_\infty$.

As is well known, there are many different proofs by many different authors for the (now called) Grothendieck's splitting theorem and one can find a discussion of the historical background in \cite{OkonekSchneiderSpindler}, p.44, where things are finally traced back to a fundamental paper of Dedekind and Weber \cite{DedekindWeber} in 1882. 

It is interesting that the proof here uses only simple results concerning finite dimensional vector spaces over a discrete valuated complete field. That we have to work in the case of genus $g=0$ is hidden in the fact, that any Laurent series from $k(T)_\infty = k((\frac{1}{T}))$ in $\frac{1}{T}$ can be uniquely written as a polynomial in $k[T]$ and a power series in $\frac{1}{T}$, vanishing at $\infty$, which in fact is equivalent to a cohomological calculation of the coherent cohomology of the structure sheaf $\mathcal{O}_{\field{P}^1}$, using the flat cover $\mathrm{Spec}(k[T]) \cup 
\mathrm{Spec}(k[[\frac{1}{T}]])$ of $\field{P}^1$ over $k$.

\section{Vector bundles on the projective line}

\subsection{Notations}
\begin{itemize}
 \item $k$ field.
 \item $ A = k[T]$ the ring of polynomials  over $k$.
 \item $ K = k(T)$ rational function field over $k$.
 \item $\infty := -\deg$ discrete valuation on $K$.
 \item $\pi := T^{-1}$ uniformizing element of $\infty$.
 \item $|\;| := \sigma^{\infty(\;)}$  associated absolute value on $K$  for some fixed real number $0 < \sigma < 1$.
 \item $ K_\infty$, a completion of $K$ with respect to this absolute value. In concrete terms the field of Laurent series in $\pi$ over $k$.
 \item $ \mathcal{O}_\infty$ valuation ring w.r.t. this absolute value in $K_\infty$. In concrete terms the ring of power series in $\pi$ over $k$.
 \item $ \mathfrak{m}_\infty = \pi \mathcal{O}_\infty$ valuation ideal.
 \item $V$ finite dimensional vector space over $K_\infty$.
 \item $\Lambda$ $A$-lattice in $V$.
 \item $\Lambda_\infty$ $\mathcal{O}_\infty$-lattice in $V$.
 \item $\lambda = \lambda^A + \lambda^\mathfrak{m}$ unique 
 decomposition   for $\lambda \in K_\infty$ s.t. $\lambda^A \in A$, $\lambda^\mathfrak{m}\in \mathfrak{m}_\infty$, the unique decomposition of $\lambda$ in its polynomial part and a part  vanishing at $ \infty$ .
 \item $< \cdots >_R$ space generated  over $R$.
\end{itemize}

\subsection{Construction of vector bundles on the projective line}

A vector bundle $\mathcal{E}$ on $\Pj^1_k$ is a locally free sheaf of finite rank $d$. If we use the faithfully flat covering $(\spec(A), \spec(\mathcal{O}_\infty))$ (c.f. \cite{Diss}, chapter 2.8)
we can describe the bundle as follows:

If we complete the stalk of $\mathcal{E}$ at the point $\infty$ we get a projective, thus free $\mathcal{O}_\infty$-module $\Lambda_\infty$ of finite rank $d$ which is canonically embedded into the $d$-dimensional $K_\infty$-vector space $V :=  K_\infty \otimes_{\mathcal{O}_\infty} \Lambda_\infty$. By construction $\Lambda_\infty$ contains a basis of $V$, so it is an $\mathcal{O}_\infty$-lattice in $V$. On the other hand, $\Lambda:=\mathcal{E}(\spec(A))$ is a projective $A$-module of rank $d$. Thus it is free as $A$ is a PID. By the glueing-data we can embed $\Lambda$ into $V$ and it contains a basis of $V$ by the glueing-condition.

Altogether we can describe a vector bundle of rank $d$ on $\Pj^1_k$ by two modules generated by bases (we call them lattices) of a $d$-dimensional $K_\infty$-vector space $V$: An $A$-lattice $\Lambda$ and an  $\mathcal{O}_\infty$-lattice $\Lambda_\infty$ inside $V$ both of rank $d$. The Theorem of Grothendieck now boils down to the question if there is a common basis for both lattices. More precisely:

\begin{Theorem}\label{Th1}
 There exist a basis $v_1, \ldots, v_d \in V$ of the $\mathcal{O}_\infty$-lattice $\Lambda_\infty$ and integers $n_1, \ldots, n_d \in \Z$ such that $\pi^{n_1}v_1, \ldots, \pi^{n_d} v_d$ is a basis of the $A$-lattice $\Lambda$. If $\mathcal{E}$ is the associated vector bundle then it splits into line bundles 
 $$\mathcal{E} \cong \mathcal{O}(n_1) \oplus \cdots \oplus \mathcal{O}(n_d).$$
\end{Theorem}

\subsection{Normed vector spaces over discrete valuated, complete fields}

Let $F$ be a discrete valuated, non-archimedean, complete field with (multiplicative) valuation $|\;|$ and $\mathring{F}$ the valuation ring.

\begin{Definition}\label{Def1}
 An $F$-vector space $V$ is a \emph{normed} vector space if there exists a norm function
 $$ || \; || : V \rightarrow \R^{\geq 0}$$
such that:
 \begin{enumerate}[1.)]
 \item $||v|| = 0 \Leftrightarrow v=0$
 \item $||\lambda v|| = |\lambda|\, ||v||$
 \item $||v+w|| \leq \max(||v||,||w||)$ (strong triangle inequality).
 \end{enumerate}
 If the set $||V\setminus \{0\}||$ is discrete in $\R$, then the norm is called discrete and if we omit condition 1) it is called a semi-norm.
\end{Definition}

\begin{Remark}\label{Rem1}
 From the strong triangle inequality we conclude:
 \begin{enumerate}[1.)]
  \item If $||v|| \neq ||w||$, then $||v-w||=||v+w|| = \max(||v||,||w||)$.
  \item If $||v|| > ||v-w||$, then $||v|| = ||w||$.
  \item A sequence $(z_n)$ in $V$ is a Cauchy-sequence iff $\lim_{n\rightarrow \infty} ||z_{n+1} - z_n || = 0$.
 \end{enumerate}
\end{Remark}

\begin{Definition}\label{Def2}
A normed vectorspace $V$ is called \emph{cartesian} if there exists a basis $\{v_1, \ldots, v_d\}$ such that 
$||v|| = \max_{1\leq i \leq d} |\lambda_i|\, ||v_i||$ for $v = \sum_{i=1}^d \lambda_i v_i$, $\lambda_i \in F$. In this case we will call the basis \emph{orthogonal}. If furthermore $||v_i|| = 1$ for $1 \leq i \leq d$, then it is called orthonormal.
\end{Definition}

\begin{Proposition}\label{Prop0}
\begin{enumerate}[1.)]
 \item If $\{v_1, \ldots, v_d\}$ is a basis of an $F$-vector space $V$, then for $v = \sum_{i=1}^d \lambda_i v_i \in V$ 
 
 $$||v|| := \max_{1\leq i \leq d} |\lambda_i| $$
 defines a norm on $V$ and $||V|| =  |F|$. Furthermore $\{v_1, \ldots, v_d\}$ is orthonormal w.r.t. $||\;||$.
 \item If $\{v_1, \ldots, v_d\}$ is an orthonormal basis of a normed $F$-vector space $V$ then 
 $$ \mathring{V}:=\{ v \in V \fdg ||v|| \leq 1\} = <v_1, \ldots, v_d>_{\mathring{F}}.$$
\end{enumerate}
\end{Proposition}

\begin{proof}
 Assertion 1.) follows from the definitions. 
 
 If $v \in \mathring{V}$ we can find $\lambda_1, \ldots, \lambda_d \in F$ such that $v = \sum_{i=1}^d \lambda_i v_i$. As $||v|| =  \max_{1 \leq i \leq d}|\lambda_i| \leq 1$ we conclude that $\lambda_1, \ldots, \lambda_d \in \mathring{F}$ and $v \in <v_1, \ldots, v_d>_{\mathring{F}}$. The other direction is obvious.
\end{proof}

\begin{Definition}\label{Def3}
 Let $W$ be a vector subspace of $V$ and $v \in V$. We define the distance
 $$||v||_W := \inf_{w \in W} ||v-w||.$$
 This defines a semi-norm on $V$ and obviously $||v||_W \leq ||v||$ as $0 \in W$.
 
 $W$ is called strictly closed  if for every $v$ in $V$ there exists $w \in W$ s.t. $||v||_W = ||v-w||$. 
\end{Definition}

\begin{Proposition}\label{Prop1}
 If $V$ is a finite dimensional normed vector space over $F$, then $V$ is complete, cartesian, $||V||$ is discrete and every subspace is strictly closed. 
 Furthermore $||V|| \subseteq |F|$ iff there exists an orthonormal basis of $V$.
\end{Proposition}

\begin{proof}
 C.f. \cite{BGR} Chapter 2.3 Proposition~4, Chapter 2.4 Proposition~1 and Proposition~3, Chapter 2.5 Observation~2.
\end{proof}

\subsection{Lattices}
Let $V$ be a $d$-dimensional normed vector space over $K_\infty$. By Proposition~\ref{Prop1} we know that $||V||$ is discrete.

\begin{Definition}
 An $A$-lattice $\Lambda$ in $V$ is a finitely generated $A$-module of rank $d$ which contains a $K_\infty$-basis of $V$. I.e. we can find $v_1, \ldots, v_d$ s.t. 
$$\Lambda = <v_1, \ldots, v_d>_A \text{ and } <v_1, \ldots, v_d>_{K_\infty} = V$$
\end{Definition}

\begin{Remark}
 It is not true in general (e.g.~if $k$ is not finite) that $\Lambda$ is discrete in the sense, that every ball contains only a finite number of lattice-points.
\end{Remark}

\begin{Proposition}\label{Prop2}
 Let $\Lambda$ be an $A$-lattice in $V$. Then there exists a shortest non-zero vector in $\Lambda$.
\end{Proposition}

\begin{proof}
 We prove a slightly different statement by induction on the dimension $d$. As $|a|\, \geq 1$ for $0 \neq a \in A$ we get our result.
 
 Let $0 \neq z_n = \sum_{i=1}^d a_i^{(n)} v_i$ be a non-trivial sequence in $V$ converging to zero and at least for one of  $1 \leq i \leq d$ the sequence $(a_i^{(n)})$ does not tend to zero. Then $\{v_1, \ldots, v_d\}$ are linearly dependent over $K_\infty$.
 
 For $d=1$ we have $z_n = a_1^{(n)} v_1$ and $v_1$ has to be zero.
 
 For $d > 1$, we choose a subsequence of $(z_n)$ and $c > 0$ s.t. for all $n$ (after renumbering the indices and dropping the other members of the sequence)
 $$|a_1^{(n)}|\, \geq \max_{1 \leq i \leq d}(c,|a_i^{(n)}|\,).$$
 We define a new sequence $\tilde{z}_n := \frac{z_n}{a_1^{(n)}}$. This sequence converges to zero because the numbers $|a_1^{(n)}|\,$ are bounded from below by $c$. 
 
 Let $y_n := \tilde{z}_{n+1} - \tilde{z}_n = \sum_{i=2}^d b_i^{(n)} v_i $ with $b_i^{(n)} := \frac{a_i^{(n+1)}}{a_1^{(n+1)}} - \frac{a_i^{(n)}}{a_1^{(n)}}$.
 Then either for all $2\leq i \leq d$ every sequence $b_i^{(n)}$ converges to zero. In this case the corresponding sequences of $\frac{a_i^{(n)}}{a_1^{(n)}}$ are converging to elements $a_i \in K_\infty$ because one has Cauchy-convergence (this follows from the strong triangle-inequality) and we get a non-trivial relation between the elements $v_1, \ldots, v_d$. Otherwise we are done by induction.
\end{proof}

\subsubsection{Successive minimum bases (SMB)}

This part is now more or less identical to Gekler's nice proof of proposition 3.1 in \cite{Gekeler}.

\begin{Proposition}\label{Prop3}
 If $\Lambda$ is an $A$-lattice in $V$, then there exists an (orthogonal) SMB $\{\omega_1, \ldots, \omega_d\}$ of $\Lambda$ e.g.
 \begin{enumerate}[1.)]
  \item $<\omega_1, \ldots, \omega_d>_A = \Lambda$,
  \item $<\omega_1, \ldots, \omega_d>_{K_\infty} = V$,
  \item $||\omega_1|| \leq \cdots \leq ||\omega_d||$,
  \item $||v|| = \max_{1 \leq i \leq d} |\lambda_i|\, ||\omega_i||$ for $v = \sum_{i=1}^d \lambda_i \omega_i$ and $\lambda_1, \ldots, \lambda_d \in K_\infty$.
 \end{enumerate}
\end{Proposition}

We construct the SMB by induction:

\begin{itemize}
 \item Choose $\omega_1 \in \Lambda$ s.t. $||\omega_1||$ is  minimal among vectors in $\Lambda \setminus \{0\}$. 
 \item If $\omega_1, \ldots, \omega_{i-1}$ are constructed, then choose $\omega_i \in \Lambda \setminus <\omega_1, \ldots, \omega_{i-1}>_{K_\infty}$, s.t. $\omega_i$ is minimal.
\end{itemize}

\begin{Remark}\label{Rem2}
 By Proposition~\ref{Prop2} and the discreteness of $||V||$ we can always find $\omega_1, \ldots, \omega_d$ as above and they are linearly independent over $K_\infty$ and
 $||\omega_1|| \leq \cdots \leq ||\omega_d||$ holds by construction. So we have to show 1.) and 4.).
\end{Remark}

\begin{Lemma}\label{Lem1}
 Let $U_i := <\omega_1, \ldots, \omega_{i}>_{K_\infty}$. Then $||\omega_i|| = ||\omega_i||_{U_{i-1}}$. 
\end{Lemma}

\begin{proof}
 For $i=1$ there is nothing to prove. For $i>1$ we only have to show that $||\omega_i|| \leq ||\omega_i||_{U_{i-1}}$ (c.f. Definition \ref{Def3}). Suppose $||\omega_i|| > ||\omega_i||_{U_{i-1}}$. As $U_ {i-1}$ is strictly closed we find $\lambda_1, \ldots, \lambda_{i-1} \in K_{\infty}$ s.t. 
 $||\omega_i||_{U_{i-1}} = || \omega_i - \sum_{j=1}^{i-1} \lambda_j \omega_j||$. Then $||\omega_i|| = || \sum_{j=1}^{i-1} \lambda_j \omega_j ||$ by Remark~\ref{Rem1}. 
 From $|| \sum_{j=1}^{i-1} \lambda^\mathfrak{m}_{j} \omega_j || < ||\omega_i||$ we see that
 $$||\omega_i - \sum_{j=1}^{i-1} \lambda^A_{j} \omega_j || \leq \max( ||\omega_i - \sum_{j=1}^{i-1} \lambda_{j}\omega_j ||, ||\sum_{j=1}^{i-1} \lambda^\mathfrak{m}_{j}\omega_j ||) < ||\omega_i||$$ 
 As $\omega_i - \sum_{j=1}^{i-1} \lambda^A_{j} \omega_j  \in \Lambda \setminus U_{i-1}$ we get a contradiction on the choice of $w_i$.
\end{proof}

\begin{Lemma}\label{Lem2}
 $\omega_1, \ldots, \omega_d$ is orthogonal.
\end{Lemma}

\begin{proof}
 We prove by induction that $\{\omega_1, \ldots, \omega_i\}$ is orthogonal for $U_i$. For $i=1$ this is clear.
 
 For $i > 1$ 
 we have to show that for $v = \sum_{j=1}^i \lambda_j \omega_j \in U_i$ one has $||v|| = \max_{1 \leq j \leq i} |\lambda_j|\, ||\omega_j||$. This is obvious if 
 $|| \sum_{j=1}^{i-1} \lambda_j \omega_j || \neq |\lambda_i|\, ||\omega_i||$ (c.f.~Remark~\ref{Rem1}). So assume 
 $$|| \sum_{j=1}^{i-1} \lambda_j \omega_j || = |\lambda_i|\, ||\omega_i||$$
 then $||v|| \leq ||\lambda_i \omega_i||$ holds. On the other hand we show that $||\lambda_i \omega_i|| \leq ||v||$:
 
 \begin{align*}
  ||\lambda_i\omega_i|| &= |\lambda_i|\, ||\omega_i|| = |\lambda_i|\, ||\omega_i||_{U_{i-1}}\\
  &= ||\lambda_i \omega_i||_{U_{i-1}} = 
 || \sum_{j=1}^{i-1} \lambda_j \omega_j + \lambda_i \omega_i||_{U_{i-1}} \leq ||v||.
 \end{align*}
\end{proof}

\begin{Lemma}\label{Lem3}
For $1\leq i \leq d$ we have $\Lambda \cap U_{i} = <\omega_1, \ldots, \omega_i>_A$. This means $\Lambda = <\omega_1, \ldots, \omega_d>_A$.
\end{Lemma}

\begin{proof}
 If $\lambda \omega_1 \in \Lambda$ then $\lambda \omega_1 = \lambda^{A} \omega_1 + \lambda^{\mathfrak{m}} \omega_1$ and $\lambda^{\mathfrak{m}} \omega_1 \in \Lambda$, so $\lambda^\mathfrak{m} =0$ because $\omega_1$ is a shortest vector in $\Lambda$. This proves the case $i=1$.
 
 If $v = \sum_{j=1}^{i} \lambda_j \omega_j \in \Lambda \cap U_i$ then as above
 $$v' = \sum_{j=1}^{i} \lambda^\mathfrak{m}_{j} \omega_j = \sum_{j=1}^{i} (\lambda_{j} - \lambda^A_{j}) \omega_j \in \Lambda.$$
 Then
 $$||v'|| = || \sum_{j=1}^{i} \lambda^\mathfrak{m}_{j} \omega_j|| < ||\omega_i||.$$
 If $\lambda^\mathfrak{m}_{i} \neq 0$, then $v' \in \Lambda \setminus U_{i-1}$ which is a contradiction to the choice of $\omega_i$. So $\lambda_i \in A$ and
 $v - \lambda_i \omega_i \in \Lambda \cap U_{i-1} = <\omega_1, \ldots, \omega_{i-1}>_A$ by induction. So we conclude, that $v \in <\omega_1, \ldots, \omega_{i}>_A$.
\end{proof}

\begin{Corollary}\label{Cor1}
 $\Lambda \cap U_i = \Lambda\; \cap <\omega_1, \ldots, \omega_i>_K = <\omega_1, \ldots, \omega_i>_A$.
\end{Corollary}

\begin{Corollary}\label{Cor2}
 We have the following characterisation of the norms of the vectors $\omega_1, \ldots, \omega_d$:
 $$||\omega_i|| = \min\{ r \in \R \fdg B_r(0) \cap \Lambda \text{ contains $i$ linearly independent vectors} \}$$ 
 This means that the numbers $||\omega_1|| \leq \cdots \leq ||\omega_d||$ are independent of the choices made in the construction.
\end{Corollary}

\begin{Lemma}\label{Lem4}
Let $\omega_1, \ldots, \omega_d \in \Lambda$ be an orthogonal basis of $V$ generating $\Lambda$ and $||\omega_1|| \leq \cdots \leq ||\omega_d||$. Then it is an SMB.
\end{Lemma}

\begin{proof}
 For $v \in \Lambda$ we have $v = \sum_{i=1}^d a_i \omega_i$ s.t.~$a_i \in A$. Then
 $$||v|| = \max_{1\leq i \leq d} |a_i|\; ||\omega_i||.$$ 
 As $|a| > 1$ for $a \in A \setminus \{0\}$ we get that $||v|| \geq ||\omega_i||$ for $v \in \Lambda \setminus U_{i-1}$ and hence the result.
\end{proof}

\subsection{Proof of Grothendieck's Theorem}

Given $\Lambda$ and $\Lambda_\infty$ in $V$, we attach to an $\mathcal{O}_\infty$-basis of $\Lambda_\infty$ a norm on $V$ s.t. $\Lambda_\infty$ becomes the unit-ball (c.f. Proposition~\ref{Prop0}). As shown above we can find an orthogonal SMB
$\omega_1, \ldots, \omega_d$ for this norm. We find integers $n_1, \ldots, n_d$ s.t. $|\pi^{n_1}| = ||\omega_1||, \ldots , |\pi^{n_d}| = ||\omega_d||$. 
Then $v_1 := \pi^{-n_1}\omega_1, \ldots,  v_d := \pi^{-n_d}\omega_d$ is an orthonormal basis for the norm and is therefore again by Proposition~\ref{Prop0} an $\mathcal{O}_\infty$-basis of $\Lambda_\infty$ and $\pi^{n_1} v_1,\ldots, \pi^{n_d} v_d$ is an $A$-basis of $\Lambda$. 
From Corollary~\ref{Cor2} and Lemma~\ref{Lem4} we conclude that $n_1, \ldots, n_d$ are unique. 

\bibliographystyle{alpha}
\bibliography{master}           

\begin{thebibliography}{BGR84}

\bibitem[BGR84]{BGR}
S.~Bosch, U.~{G\"untzer}, and R.~Remmert.
\newblock {\em Non-{A}rchimedean analysis}, volume 261 of {\em Grundlehren der
  Mathematischen Wissenschaften [Fundamental Principles of Mathematical
  Sciences]}.
\newblock Springer-Verlag, Berlin, 1984.
\newblock A systematic approach to rigid analytic geometry.

\bibitem[DW82]{DedekindWeber}
R.~Dedekind and H.~Weber.
\newblock Theorie der algebraischen {F}unctionen einer {V}er\"anderlichen.
\newblock {\em J. Reine Angew. Math.}, 92:181--290, 1882.

\bibitem[Gek17]{Gekeler}
Ernst-Ulrich Gekeler.
\newblock Towers of $\mathrm{GL}(r)$-type of modular curves.
\newblock {\em J. Reine Angew. Math.}, 2017.

\bibitem[Mar03]{Martinet}
Jacques Martinet.
\newblock {\em Perfect lattices in {E}uclidean spaces}, volume 327 of {\em
  Grundlehren der Mathematischen Wissenschaften [Fundamental Principles of
  Mathematical Sciences]}.
\newblock Springer-Verlag, Berlin, 2003.

\bibitem[OSS80]{OkonekSchneiderSpindler}
Christian Okonek, Michael Schneider, and Heinz Spindler.
\newblock {\em Vector bundles on complex projective spaces}, volume~3 of {\em
  Progress in Mathematics}.
\newblock Birkh\"auser, Boston, Mass., 1980.

\bibitem[Tag93]{Taguchi}
Yuichiro Taguchi.
\newblock Semi-simplicity of the {G}alois representations attached to
  {D}rinfel'd modules over fields of {``}infinite characteristics{''}.
\newblock {\em J. Number Theory}, 44(3):292--314, 1993.

\bibitem[Wie05]{Diss}
Stefan Wiedmann.
\newblock {\em Drinfeld modules and elliptic sheaves}.
\newblock PhD thesis, Mathematical Institute, University of Goettingen, 2005.

\end{thebibliography}
\end{document}